\documentclass[12pt,reqno]{amsart} 

\usepackage{amssymb,bbm,enumitem,url}
\usepackage{aliascnt}
\usepackage{doi}
\usepackage{pgfplots} 
\pgfplotsset{compat=1.12}
\usepackage{esint}

\usepackage[margin=1.25in]{geometry}

\newtheorem{theorem}{Theorem}

\newaliascnt{lemma}{theorem}
\newtheorem{lemma}[lemma]{Lemma}
\aliascntresetthe{lemma}

\newaliascnt{proposition}{theorem}

\aliascntresetthe{proposition}

\newaliascnt{corollary}{theorem}
\newtheorem{corollary}[corollary]{Corollary}
\aliascntresetthe{corollary}

\newaliascnt{conjecture}{theorem}

\aliascntresetthe{conjecture}

\newaliascnt{openQ}{theorem}

\aliascntresetthe{openQ}

\newaliascnt{quest}{theorem}

\aliascntresetthe{quest}

\newaliascnt{questx}{conjx}

\aliascntresetthe{questx}

\theoremstyle{definition}

\newaliascnt{defn}{theorem}

\aliascntresetthe{defn}

\newaliascnt{example}{theorem}

\aliascntresetthe{example}

\newaliascnt{rem}{theorem}

\aliascntresetthe{rem}

\makeatletter
\def\tagform@#1{\maketag@@@{\ignorespaces#1\unskip\@@italiccorr}}
\let\orgtheequation\theequation
\def\theequation{(\orgtheequation)}
\makeatother
\def\equationautorefname~{}

\newcommand{\e}{\varepsilon}
\newcommand{\R}{{\mathbb R}}
\newcommand{\Rn}{{{\mathbb R}^n}}

\begin{document}
\title[Center of mass --- euclidean well-posedness]{Well-posedness of Weinberger's center of mass by euclidean energy minimization}


	\begin{abstract}
The center of mass of a finite measure with respect to a radially increasing weight is shown to exist, be unique, and depend continuously on the measure. 
	\end{abstract}
	
\author[]{R. S. Laugesen}
\address{Department of Mathematics, University of Illinois, Urbana,
	IL 61801, U.S.A.}
\email{Laugesen@illinois.edu}

	\maketitle
	
\begin{center}
\emph{Dedicated to Guido Weiss, with gratitude for his encouragement, and appreciation of his far-reaching vision in Analysis.}
\end{center}

\section{\bf Introduction} \label{sec:intro}

\subsection*{Motivation} 
The center of mass of a finite, compactly supported measure $\mu$ on $\Rn$ is the point $c$ for which $\int (y-c) \, d\mu(y) = 0$. The formula  $c=\fint y \, d\mu(y)$ shows that the center of mass exists, is unique, and depends continuously on $\mu$. That is, the center of mass is well-posed. This paper establishes well-posedness for generalized centers of mass that arise when proving sharp upper bounds on eigenvalues of the Laplacian. 

Consider a radial weight $g(r)$ with $g(0)=0$, as illustrated in \autoref{fig:ggraph}. The task is to find conditions on $g$ and the measure $\mu$ under which the generalized center of mass equation 
\begin{equation} \label{eq:gcenter}
\int_\Rn g(|x+y|) \frac{x+y}{|x+y|} \, d\mu(y) = 0 
\end{equation}
has a solution $x \in \Rn$, and to determine when this point $x$ is unique and depends continuously on $\mu$. Notice that choosing $g(r)=r$ in condition \eqref{eq:gcenter}, and writing $x=-c$, reduces it to the traditional center of mass equation.
\begin{figure}[t]
\begin{center}
\includegraphics[scale=0.35]{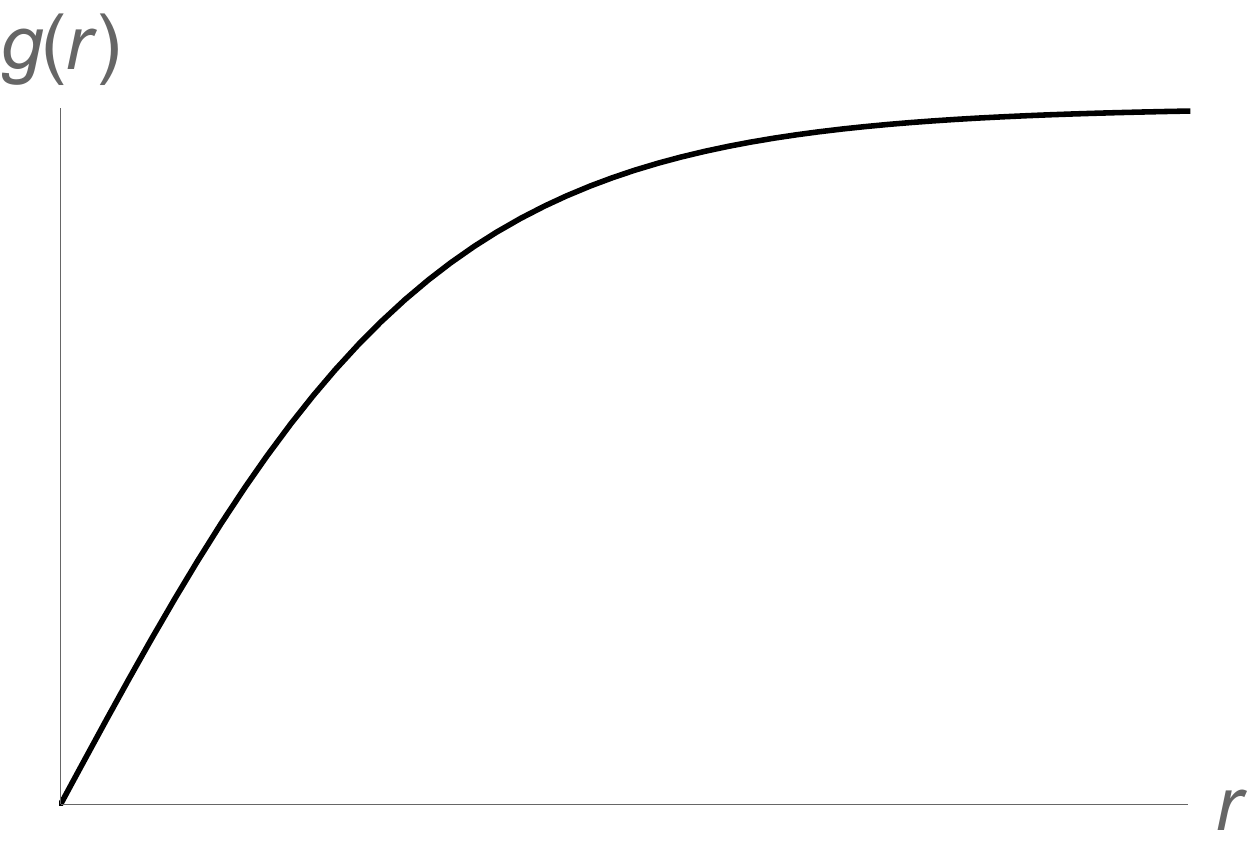} \qquad \includegraphics[scale=0.35]{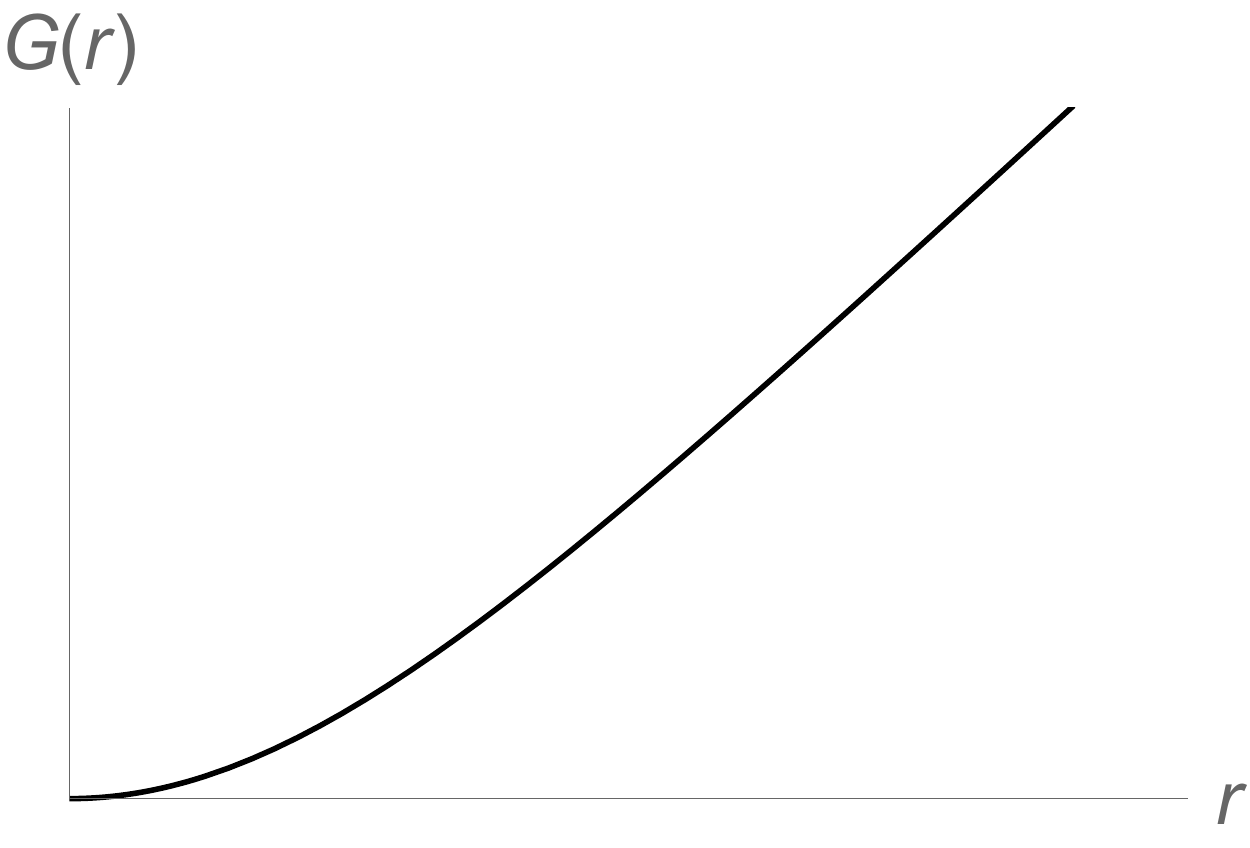}
\end{center}
\caption{\label{fig:ggraph}\textsc{Left:} A radial weight $g(r)$, with $g(0)=0$. The existence results in this paper do not assume $g$ to be nonnegative or increasing. The uniqueness and continuous dependence results do assume $g$ to be positive and increasing. \textsc{Right:} The energy kernel $G$ is the antiderivative of $g$, and so $G$ is convex if $G^\prime=g$ is increasing, as in the example shown.}
\end{figure}

Results for the $g$-center of mass have been driven by the applications at hand. The measure $\mu$ is typically taken to be a density times Lebesgue measure, on some bounded domain, or else $\mu$ is surface area measure on the boundary. The weight $g$ is usually increasing, and is constant for large $r$. The current paper assumes less about $g$, and handles arbitrary finite measures and allows $\mu$ to have unbounded support. 

Spectral applications in euclidean space that require the $g$-center of mass started with Weinberger \cite{W56}, whose work maximizing the second Neumann eigenvalue on a bounded domain provided a foundation for Ashbaugh and Benguria \cite{AB92}  when they maximized the ratio of the first two Dirichlet eigenvalues (the sharp PPW conjecture). Brock \cite{B01} treated the second Steklov eigenvalue, for which $\mu$ is surface area measure on the boundary. Omitting many further contributions over the decades, we arrive at a recent paper by Bucur and Henrot \cite{BH19} using center of mass results to maximize the third Neumann eigenvalue. The $g$-center of mass remains of enduring importance. 

\subsection*{Overview of results} 

\autoref{th:weinberger} proves well-posedness of the generalized center of mass for compactly supported measures, assuming for existence that $\int_0^\infty g(r) \, dr = \infty$, and assuming for uniqueness and continuous dependence that $g$ is increasing. Slightly more will be assumed, in fact, because uniqueness can fail if $g$ is non-strictly increasing and $\mu$ is supported in a line. 

\autoref{cor:wein} deduces the original application by Weinberger, in which $\mu$ is Lebesgue measure on a bounded domain. A ``folded'' variant in \autoref{cor:weinfold} provides an alternative viewpoint on a tool used by Bucur and Henrot \cite{BH19}. 

Measures of unbounded support are treated in \autoref{th:weinbergerallspace}, getting well-posedness. If one only wants the existence of a center of mass point, then one may relax the hypotheses to consider signed measures (\autoref{th:weinbergersigned}). 

\subsection*{Methods} The classical center of mass is found by minimizing the moment of inertia $\int |y-c|^2 \, d\mu(y)$ with respect to the choice of center point $c$. The analogous quantity to minimize for the $g$-center of mass is 
\[
E(x) = \int_\Rn G(|x+y|) \, d\mu(y) , \qquad x \in \Rn , 
\]
where $G^\prime = g$. With some poetic license and abuse of physics, we call $E$ an \emph{energy}. Its gradient is precisely the vector field on the left side of \eqref{eq:gcenter}, and so critical points of the energy (in particular, any minimum points) are automatically centers of mass. For existence of an energy minimizing point one wants to show that the energy tends to infinity as $|x| \to \infty$, while for uniqueness and continuous dependence one wants the energy to be strictly convex.

The energy $E$ is the correct tool when the measure $\mu$ has compact support, as in \autoref{th:weinberger}. Measures with noncompact support are handled in \autoref{th:weinbergerallspace} by utilizing a renormalized energy
\[
\mathcal{E}(x) = \int_\Rn \big( G(|x+y|) - G(|y|) \big) \, d\mu(y) , \qquad x \in \Rn , 
\]
whose kernel extends continuously to the sphere at infinity (see \autoref{le:renormkerneleuclid}).

\subsection*{Prior results}
The traditional method for proving center of mass results, which goes back to Weinberger \cite{W56} (with conformal mapping antecedents in  Szeg\H{o} \cite{S54}), consists of showing that the left side vector field in \eqref{eq:gcenter} points outward when $x$ lies on a sphere of large radius. Then by Brouwer's fixed point theorem, the vector field must vanish somewhere inside the sphere, giving a $g$-center of mass point. This index theory argument does not, by itself, seem capable of proving uniqueness or continuous dependence, which is why the current paper employs exclusively the energy method. 

The energy method for proving existence of the $g$-center of mass was presented by Brasco and Franzina \cite[Lemma 7.1]{BF13} (also Brasco and De Philippis \cite[{\S}7.4.3]{BdP17}). The method was known independently to Ashbaugh in the early 2000s (unpublished). I learned it from him and Langford in conversation some years ago.  

Ashbaugh and Benguria \cite[p.~407]{AB95} pointed out that the Brouwer fixed point method for existence could be applied to a general measure $\mu$. Bucur and Henrot \cite{BH19} applied the Brouwer approach in a euclidean situation with $\mu$ a weighted Lebesgue measure folded across a hyperplane. They also proved uniqueness: \autoref{cor:weinfold} explains their result. Bucur and Henrot further proved existence for noncompactly supported densities; see \autoref{cor:weinRn} and the remarks following it. Densities with noncompact support were treated earlier by Aubry, Bertrand and Colbois \cite[Lemma 4.11]{ABC09}.

Finally, well-posedness of the (conformal) center of mass for a finite measure in the $2$-dimensional unit disk was proved by Girouard, Nadirashvili and Polterovich \cite[Lemmas 2.2.3--2.2.5, 3.1.1]{GNP09} using Szeg\H{o}-type methods and some ingenious estimates in the disk. Their work provides inspiration for the current paper on well-posedness in euclidean space.  

\subsection*{Summary of what is new in this paper, and what lies ahead} 
The energy method in this paper provides a powerful and flexible template for proving existence, uniqueness and continuous dependence of the weighted center of mass.  

The theorems are developed for finite measures. This level of generality requires some care in the uniqueness statements, compared to when the measure is given by a density times Lebesgue measure, because uniqueness can fail if $g$ is non-strictly increasing and the measure $\mu$ is supported in a line. 

Measures with unbounded support are treated in this paper by renormalizing the energy, which seems preferable to earlier approaches involving approximation or truncation together with passing to limits. The use of the renormalized energy, and the uniqueness and continuous dependence results that follow from it for measures of unbounded support, are new to the best of my knowledge. 

The energy methods in this paper not only prove existence of the $g$-center of mass, they suggest that one may compute it numerically by applying a steepest descent or Newton algorithm to converge to an energy minimum. Such numerical methods would be particularly efficient when $g$ is increasing, since then the energy is convex and has just a single global minimum. In contrast, the Brouwer fixed point approach for proving existence of a center of mass does not suggest a practical method for finding it. 

The renormalized energy can be adapted to the hyperbolic ball, where the role of translations is played by M\"{o}bius isometries and the boundary sphere at infinity can be identified with the unit euclidean sphere. This renormalized hyperbolic energy approach will be developed in a subsequent paper \cite{L20h} to obtain well-posedness of hyperbolic centers of mass. Corollaries in that paper include both the conformal center of mass result of Szeg\H{o} \cite{S54} in the disk, and the center of mass normalization of Hersch \cite{H70} on the sphere.

\section{\bf Results}  \label{sec:resultseuclidean}

Assume throughout the paper that 
\[
\text{$g(r)$ is continuous and real valued for $r \geq 0$, with $g(0)=0$,}
\]
and $\mu$ is a Borel measure on $\Rn, n \geq 1$, with
\[
0 < \mu(\Rn) < \infty .
\]
A typical radial profile $g$ is shown in \autoref{fig:ggraph}, although not all our results will assume $g$ is nonnegative  increasing like the example shown. 

Define $v : \Rn \to \Rn$ to be the radial vector field with magnitude $g$, meaning
\[
v(y) = g(|y|) \frac{y}{|y|} , \qquad y \in \Rn \setminus \{ 0 \} ,
\]
and $v(0)=0$. In other words, $v(r\hat{y})=g(r)\hat{y}$ whenever $r \geq 0$ and $\hat{y}$ is a unit vector. Notice $v$ is continuous at the origin, since $g(0)=0$. 

The vector field 
\[
V(x) = \int_\Rn v(x+y) \, d\mu(y) , \qquad x \in \Rn ,
\]
which arises by integrating translates of $v$, is well defined if $\mu$ has compact support or if $g$ (and hence $v$) is bounded. We seek a point $x_c$ for which $V(x_c)=0$. Its antipodal point $-x_c$ then gives a $g$-center of mass for $\mu$. 

The first theorem establishes conditions under which the center of mass exists, is unique, and depends continuously on the measure $\mu$. 
\begin{theorem}[Center of mass for compactly supported measures] \label{th:weinberger}\ \\ 
Assume the measure $\mu$ has compact support.

\noindent (a) [Existence] If $\int_0^\infty g(r) \, dr = \infty$ then $V(x_c)=0$ for some $x_c \in \Rn$. 

\noindent (b) [Uniqueness] If either 
\begin{enumerate}[label=(\roman*),nosep]
\item $g$ is strictly increasing, or
\item $g$ is increasing, $g(r)>0$ for all $r>0$, and $\mu$ is not supported in a line, 
\end{enumerate}
then the point $x_c$ is unique.  

\noindent (c) [Continuous dependence] Suppose $\mu_k \to \mu$ weakly, where the $\mu_k$ are Borel measures all supported in a fixed compact set in $\Rn$ and satisfying $0 < \mu_k(\Rn) < \infty$. If either (i) holds or else (ii) holds for $\mu$ and each $\mu_k$, then $x_c(\mu_k) \to x_c(\mu)$ as $k \to \infty$. 
\end{theorem}
The proof is in \autoref{sec:weinbergerproof}. The hypothesis $\int_0^\infty g(r) \, dr = \infty$ in part (a) should be interpreted in terms of improper integrals, saying $\int_0^\rho g(r) \, dr \to \infty$ as $\rho \to \infty$. In part (b)(ii), to say $\mu$ is not supported in a line means $\mu(\Rn \setminus L)>0$ for every line $L$ in $\Rn$.

Uniqueness can fail in part (b)(ii) when the measure $\mu$ is supported in a line. The phenomenon occurs already in dimension $n=1$: take $g(r) = \min(r,1)$, so that $g$ increases from $0$ to $1$ for $r \in [0,1]$ and is constant for $r \in [1,\infty)$, and suppose $\mu=\delta_a+\delta_b$ is a sum of point masses at locations $a$ and $b$ with $a<-1<1<b$. Then 
\[
V(x) = v(x+a) + v(x+b) = g(|x+a|)  \cdot (-1) + g(|x+b|) \cdot 1 = -1+1=0
\]
whenever $x$ is close enough to $0$ that $x+a < -1$ and $x+b > 1$. Thus $V$ vanishes for a whole interval of $x$ values, and so uniqueness fails rather dramatically. 

Continuous dependence can fail in part (c) when the supports of the measures are not all contained in a fixed compact set. For example, consider in $1$ dimension the measure $\mu_k=(1-1/\sqrt{k})\delta_0+\delta_k/\sqrt{k}$. Its traditional center of mass (coming from $g(r)=r$) sits at $-x_c(\mu_k)=\sqrt{k}$, and so runs off to infinity as $k \to \infty$, even though $\mu_k$ converges weakly to $\mu=\delta_0$, whose center of mass is at the origin. 

The ``fixed compact set'' assumption on the measures in part (c) can be dropped if $g$ is bounded. For this, see \autoref{th:weinbergerallspace}(c). 
\begin{corollary}[Weinberger's orthogonality] \label{cor:wein}
Suppose $\Omega$ is a bounded domain in $\Rn$ and $f$ is nonnegative and integrable on $\Omega$ with $\int_\Omega f(y) \, dy > 0$. If $\int_0^\infty g(r) \, dr = \infty$ then a point $x \in \Rn$ exists such that each component of the vector field $v(x+\cdot)$ is orthogonal to $f$, meaning
\[
\int_\Omega v(x+y) f(y) \, dy = 0 .
\]
If in addition $g$ is increasing with $g(r)>0$ for all $r>0$ then the point $x$ is unique.  
\end{corollary}
\begin{proof} 
Apply \autoref{th:weinberger} parts (a) and (b)(ii) with $d\mu(y) = f(y) \, dy|_\Omega$. Clearly this measure $\mu$ is not supported in any line.
\end{proof}
Weinberger \cite[p.~635]{W56} proved the existence statement of the corollary. (He used $f \equiv 1$, but the general argument is the same.) The uniqueness statement was shown by Bucur and Henrot \cite[Lemmas 5 and 6]{BH19} for a $g$ that is increasing and is constant for $r \geq R$. Their Lemma 5 is not quite correct as stated, because its strict inequality  must actually be an equality when $x$ lies on the line passing through their points $A$ and $B$ with $x$ having distance at least $R$ to both of those points. The set of such $x$ has measure zero, though, and so the deduction of uniqueness in their Lemma 6 remains valid.

For a more modern application of the theorem, let $H$ be a closed halfspace in $\Rn$, and define $F : \Rn \to H$ to be the ``fold map'' that fixes each point in $H$ and maps each point in $\Rn \setminus H$ to its reflection across the hyperplane $\partial H$.  
\begin{corollary}[Orthogonality with a fold] \label{cor:weinfold}
Suppose $\Omega$ is a bounded domain in $\Rn$ and $f$ is nonnegative and integrable on $\Omega$ with $\int_\Omega f(y) \, dy > 0$. If $\int_0^\infty g(r) \, dr = \infty$ and the halfspace $H$ and its fold map $F$ are given, then a point $x \in \Rn$ exists such that each component of the vector field $v(x+F(\cdot))$ is orthogonal to $f$, meaning
\[
\int_\Omega v(x+F(y)) f(y) \, dy = 0 .
\]
If in addition $g$ is increasing with $g(r)>0$ for all $r>0$ then the point $x=x(H)$ is unique and depends continuously on the halfspace $H$.  
\end{corollary}
The proof can be found in \autoref{sec:weinfoldproof}. The corollary, when applied with $f \equiv 1$, gives orthogonality of the constant eigenfunction to a folded copy of the vector field $v$. This orthogonality is due to Bucur and Henrot \cite[part of Proposition 10]{BH19}. The corollary does not address the more difficult part of their proposition, which simultaneously achieves orthogonality with respect to the first nonconstant eigenfunction, by means of a subtle homotopy argument that uncovers a good choice for the halfspace $H$. Incidentally, Bucur and Henrot formulated their construction somewhat differently, in terms of gluing rather than folding. 

Next we allow measures of unbounded support, provided $g$ is bounded, which we did not need to assume in \autoref{th:weinberger} because the measure there had compact support. Write
\[
g(\infty) = \lim_{r \to \infty} g(r)
\]
for the limiting value of $g$ at infinity, when that limit exists. 

\begin{theorem}[Center of mass for arbitrary finite measures] \label{th:weinbergerallspace}\ 

\noindent (a) [Existence] If $g$ has a positive and finite limit at infinity, $0 < g(\infty) < \infty$, then $V(x_c)=0$ for some $x_c \in \Rn$. 

\noindent (b) [Uniqueness] If either 
\begin{enumerate}[label=(\roman*),nosep]
\item $g$ is strictly increasing and bounded, or
\item $g$ is increasing and bounded, with $g(r)>0$ for all $r>0$, and the measure $\mu$ is not supported in a line, 
\end{enumerate}
then the point $x_c$ is unique.  

\noindent (c) [Continuous dependence] Suppose $\mu_k \to \mu$ weakly, where the $\mu_k$ are Borel measures satisfying $0 < \mu_k(\Rn) < \infty$ for all $k$. If either (i) holds or else (ii) holds for $\mu$ and each $\mu_k$, then $x_c(\mu_k) \to x_c(\mu)$ as $k \to \infty$. 
\end{theorem}
The proof is in \autoref{sec:weinbergerproofallspace}. The weak convergence hypothesis in part (c) means that $\int_\Rn \psi \, d\mu_k \to \int_\Rn \psi \, d\mu$ as $k \to \infty$ for each bounded continuous function $\psi$ on $\Rn$. 

The existence claim in \autoref{th:weinbergerallspace}(a) holds even when $\mu$ is a \textbf{signed} measure, as the next result shows. 
\begin{theorem}[Center of mass for a signed measure --- existence] \label{th:weinbergersigned} Suppose $\mu$ is a finite signed Borel measure on $\Rn$ with 
\[
0 < \mu(\Rn) \leq |\mu|(\Rn) < \infty .
\]
If $g$ has a positive and finite limit at infinity, $0 < g(\infty) < \infty$, then $V(x_c)=0$ for some $x_c \in \Rn$. 
\end{theorem}
See \autoref{sec:weinbergerproofsigned} for the proof. The theorem makes no claims about uniqueness, because uniqueness can fail for signed measures, by the following $1$-dimensional example. Let 
\[
\mu=-\delta_{-1}+3\delta_0-\delta_1 ,
\] 
so that $\mu$ consists of negative point masses at $x=\pm 1$ and a triple point mass at the origin, giving $\mu(\R)=1>0$. Choosing 
\[
g(r) = 
\begin{cases}
r , & 0 \leq r \leq 1, \\
2r-1 & 1 \leq r \leq 2 , \\
3 & r \geq 2 ,
\end{cases}
\]
we compute that $v(-1)=-1,v(0)=0,v(1)=1,v(2)=3$ and hence 
\[
V(0)=(-1)(-1)+0 \cdot 3+1(-1)=0 , \qquad V(1)=0(-1)+1 \cdot 3 +3(-1)=0 .
\]
Thus $V$ vanishes at more than one point, In fact, one can check that $V(x)=0$ for all $x \in [-1,1]$, and so uniqueness fails badly. 

Finally, we specialize the last two theorems to sign-changing densities that may have unbounded support. 
\begin{corollary}[Weinberger's orthogonality for signed densities] \label{cor:weinRn}
If $f$ is real-valued and integrable on $\Rn$ with $\int_\Rn f(y) \, dy > 0$ and $g$ has a positive and finite limit at infinity, $0 < g(\infty) < \infty$, then a point $x \in \Rn$ exists such that each component of the vector field $v(x+\cdot)$ is orthogonal to $f$, meaning
\[
\int_\Rn v(x+y) f(y) \, dy = 0 .
\]
If $f$ is nonnegative with $0 <\int_\Rn f(y) \, dy < \infty$ and $g$ is increasing and bounded with $g(r)>0$ for all $r>0$, then the point $x$ is unique.  
\end{corollary}
\begin{proof} 
Put $d\mu(y) = f(y) \, dy$. Apply \autoref{th:weinbergersigned} for the existence claim, and \autoref{th:weinbergerallspace}(b)(ii) for  uniqueness, noting $\mu$ is not supported in any line.
\end{proof}
The existence assertion in the corollary was proved by Bucur and Henrot \cite[page 355]{BH19}, for nonnegative $f$ and functions $g(r)$ that are increasing and constant for all large $r$.

\section{\bf Proof of \autoref{th:weinberger} --- \textit{g}-center of mass for compactly supported measures}  \label{sec:weinbergerproof}	

Existence of a vanishing point for $V$ will follow from expressing $V$ as the gradient of an energy functional that grows to infinity. Uniqueness is a consequence of strict convexity of the energy, which we establish by two methods, one analytic in nature and the other more geometric; both approaches have their appeal, although the geometric method offers perhaps  more insight. Continuous dependence then follows from uniqueness and a compactness argument.  

\subsection*{Part (a) --- Existence} Let 
\[
G(r)=\int_0^r g(s) \, ds
\]
be the antiderivative of $g$ with $G(0)=0$, and put
\[
\Gamma(x) = G(|x|) ,  \qquad x \in \Rn ,
\]
so that $\nabla \Gamma (x) = g(|x|) x/|x| = v(x)$. (The equation $\nabla \Gamma = v$ continues to hold at $x=0$, since $g(0)=0$ implies $\Gamma(x)=o(|x|)$ and so $\nabla \Gamma(0)=0$, while $v(0)=0$ by definition.) Define an energy functional
\begin{align*}
E(x) 
& = \int_\Rn \Gamma(x+y) \, d\mu(y) \\
& = \int_\Rn G(|x+y|) \, d\mu(y) , \qquad x \in \Rn .
\end{align*}
Notice $E$ is finite-valued and depends continuously on $x$, since $G$ is continuous, $\mu$ has compact support, and $\mu$ is a finite measure. Further, $E(x) \to \infty$ as $|x| \to \infty$, because $G(\infty) = \int_0^\infty g(s) \, ds = \infty$ by assumption and $\mu$ is compactly supported with $\mu(\Rn)>0$. Hence $E$ achieves a minimum at some point $x_c$. 

At this minimum point the gradient must vanish, and so by differentiating through the integral, 
\[
0 = (\nabla E)(x_c) = \int_\Rn v(x_c+y) \, d\mu(y) = V(x_c) ,
\]
thus proving existence of a point at which $V$ vanishes. 

\smallskip
\subsection*{Part (b) --- Uniqueness by geometric convexity} 
Conditions (i) and (ii) each imply that $g(r)$ is increasing and positive for $r>0$. Hence $\int_0^\infty g(r) \, dr = \infty$, and so part (a) guarantees the existence of a critical point $x_c$ at which $V=\nabla E$ vanishes. 

We will show later in the proof that:
\begin{align} 
& \text{if condition (i) holds then the kernel $\Gamma$ is strictly convex;} \label{eq:convexitycondi} \\
& \text{if condition (ii) holds then $\Gamma$ is convex along each straight line, and} \notag \\
& \text{the convexity is strict if the line does not pass through the origin.} \label{eq:convexitycondii}
\end{align}
Assuming these facts for now, if condition (i) holds then $x \mapsto \Gamma(x+y)$ is strictly convex by \eqref{eq:convexitycondi}, for each $y \in \Rn$. Hence integrating with respect to $d\mu(y)$ gives strict convexity of $E(x)$, and so its critical point $x_c$ is unique. Similarly, if condition (ii) holds, then \eqref{eq:convexitycondii} gives convexity of $E(x)$ along each straight line $\gamma$, and the convexity is strict unless the line $\gamma+y$ passes through the origin for $\mu$-almost every $y \in \Rn$. That exceptional case would imply $\gamma$ contains the point $-y$ for $\mu$-almost every $y$, and so $\mu$ would be supported in the line $-\gamma$. But condition (ii) assumes $\mu$ is not supported in any line. Therefore $E(x)$ is strictly convex along each line $\gamma$, implying uniqueness of the critical point $x_c$.

\smallskip
It remains to prove implications \eqref{eq:convexitycondi} and \eqref{eq:convexitycondii}. The first step is to show that if $g$ is increasing and $g(r)>0$ for all $r>0$ (which holds under both assumptions (i) and (ii)) then $\Gamma $ is convex on $\Rn$. Notice $G$ is strictly increasing since $G^\prime=g>0$, and $G$ is convex since $G^\prime=g$ is increasing. Consider $x,\hat{x} \in \Rn$ with $x \neq \hat{x}$, let $0<\e<1$, and observe 
\begin{align}
\Gamma((1-\e)x+\e \hat{x}) 
& = G \big( |(1-\e)x+\e \hat{x} | \big) \notag \\
& \leq G \big( (1-\e)|x|+\e |\hat{x}| \big) \label{eq:convexGamma} \\
& \hspace*{0.8cm} \text{by the triangle inequality and since $G$ is increasing} \notag \\
& \leq (1-\e) G(|x|) + \e G(|\hat{x}|) \qquad \text{by convexity of $G$} \label{eq:convexGamma2}  \\
& = (1-\e) \Gamma(x) + \e \Gamma(\hat{x}) . \notag
\end{align}
Hence $\Gamma$ is convex. Further, since $G$ is strictly increasing, equality holds in \eqref{eq:convexGamma} if and only if equality holds in the triangle inequality, which occurs when the vectors $x$ and $\hat{x}$ point in the same direction.  

Suppose condition (i) holds, so that $G^\prime = g$ is strictly increasing and hence $G$ is strictly convex. If equality holds in \eqref{eq:convexGamma2} then the strict convexity of $G$ implies $|x|=|\hat{x}|$, and so the vectors $x$ and $\hat{x}$ have the same magnitude. Since $x \neq \hat{x}$ by assumption, they must point in different directions, and so inequality \eqref{eq:convexGamma} is strict. Hence $\Gamma$ is strictly convex, proving implication \eqref{eq:convexitycondi}. 

Now suppose condition (ii) holds. To prove \eqref{eq:convexitycondii} we must show that if the convexity of $\Gamma$ along some straight line is not strict, then that line passes through the origin. For this, observe that if equality holds in \eqref{eq:convexGamma} for some $x \neq \hat{x}$ then the points $x$ and $\hat{x}$ must lie on some ray from the origin, and so the line through those points must also pass through the origin.

\smallskip
\subsection*{Part (b) --- Uniqueness by analytic convexity} 
Just as in the geometric proof above, the task reduces  to establishing convexity of the kernel $\Gamma$, that is, to proving  implications \eqref{eq:convexitycondi} and \eqref{eq:convexitycondii}. This time we take a more analytic approach. 

Consider an arbitrary line $x(t)=a+bt$ (where $a,b \in \Rn, |b|=1$). The derivative of $\Gamma$ along the line is
\[
\frac{d\ }{dt} \Gamma \big( x(t) \big) = g\big( |x(t)| \big) \frac{d\ }{dt} |x(t)| = v\big(x(t)\big) \cdot x^\prime(t) ,
\]
where we used that $G^\prime=g$. The convexity implications \eqref{eq:convexitycondi} and \eqref{eq:convexitycondii} to be proved can be rewritten as:
\begin{align} 
\text{if condition (i) holds then $v\big(x(t)\big) \cdot x^\prime(t)$ is strictly increasing;} \label{eq:convexitycondiii} \\
\text{if condition (ii) holds then $v\big(x(t)\big) \cdot x^\prime(t)$ is increasing, and is} \notag \\
\text{strictly increasing if the line does not pass through the origin.} \label{eq:convexitycondiv}
\end{align}

First suppose $g$ is increasing and $g(r)>0$ for all $r>0$, which holds under both conditions (i) and (ii). Suppose further that the line does not pass through the origin. We will show $v\big(x(t)\big) \cdot x^\prime(t)$ is a strictly increasing function of $t \in \R$. Indeed, $t \mapsto |a+bt|$ is strictly convex, as can be deduced easily from the triangle inequality. Write $t_{min}$ for the value at which $|a+bt|$ is minimal, so that $|a+bt|$ is positive and decreasing for $t<t_{min}$ and is positive and increasing for $t>t_{min}$. Hence $g\big( |a+bt| \big)$ has the same properties, because $g(r)$ is positive and increasing for $r>0$. Further, the derivative $(d/dt) |a+bt|$ is negative and strictly increasing for $t<t_{min}$, and positive and strictly increasing for $t>t_{min}$, by the strict convexity. Putting these facts together shows that shows that $g\big( |x(t)| \big) \frac{d\ }{dt} |x(t)|=v\big(x(t)\big) \cdot x^\prime(t)$ is a strictly increasing function of $t$. This proves \eqref{eq:convexitycondiii} and \eqref{eq:convexitycondiv} when the line does not pass through the origin. 

Suppose the line does pass through the origin. Then $|a+bt|=|t-t_{min}|$ and so $v\big(x(t)\big) \cdot x^\prime(t) = \operatorname{sign}(t-t_{min}) g(|t-t_{min}|)$, which is increasing for $t \in \R$ and is strictly increasing if $g$ is strictly increasing. This finishes the proof of \eqref{eq:convexitycondiii} and \eqref{eq:convexitycondiv}.  

\smallskip
\subsection*{Part (c) --- Continuous dependence} 
Part (c) assumes that either (i) holds or else (ii) holds for $\mu, \mu_1,\mu_2,\mu_3,\dots$. These conditions imply $\int_0^\infty g(r) \, dr = \infty$. Thus the hypotheses of parts (a) and (b) are satisfied, with respect to the measures $\mu$ and $\mu_k$. 
Write $x_c(\mu)$ for the unique minimum point of the energy $E$ corresponding to the measure $\mu$, and $x_c(\mu_k)$ for the unique minimum point of the energy $E_k$ corresponding to the measure $\mu_k$. 

The measures $\mu_k$ and $\mu$ are assumed to be supported in some fixed compact set $Y$, and the weak convergence $\mu_k \to \mu$ implies that $\mu_k(Y) \to \mu(Y)$. Let $X$ be an arbitrary compact set in $\Rn$. The kernel $(x,y) \mapsto \Gamma(x+y)$ is uniformly continuous and bounded on $X \times Y$, and it follows easily that the family $\{ E_k(x) \}_{k=1}^\infty$ is uniformly equicontinuous on $X$. Therefore the weak convergence $\mu_k \to \mu$ implies that $E_k(x) \to E(x)$ pointwise and (after a short argument using equicontinuity) uniformly on $X$. 
%
%

Let $\e>0$, and denote by $B$ the open ball of radius $\e$ centered at $x_c(\mu)$. The strict minimizing property of $x_c(\mu)$ implies 
\[
E(x_c(\mu)) < \min_{x \in \partial B} E(x) ,
\]
and so (by choosing $X=\partial B$) we deduce 
\[
E_k(x_c(\mu)) < \min_{x \in \partial B} E_k(x) 
\]
for all large $k$. Consequently, the open ball $B$ contains a local minimum point for the energy $E_k$. This local minimum must be the global minimum point $x_c(\mu_k)$, by strict convexity of the energy. Since $\e$ was arbitrary, we conclude $x_c(\mu_k) \to x_c(\mu)$ as $k \to \infty$, giving continuous dependence.

\section{\bf Proof of \autoref{cor:weinfold} --- orthogonality with a fold}  \label{sec:weinfoldproof}

For existence, apply \autoref{th:weinberger} to the measure $d\mu = (f \, dy) \circ F^{-1}$, that is, with $\mu$ being the pushforward under the fold map $F$ of the measure $f(y) \, dy$ on $\Omega$. Since $\mu$ is not supported on any line, condition (ii) holds in part (b) of the theorem, giving uniqueness. 

To obtain continuous dependence, we must verify the hypotheses of part (c) of the theorem. Write the halfspace as $H(p,t) = \{ y \in \Rn :  y \cdot p \leq t \}$ where $t \in \R$ and the normal vector is $p \in S^{n-1}$. Suppose $p_k \to p$ in $S^{n-1}$ and $t_k \to t$ in $\R$. Write $F_k$ for the fold map associated with the halfspace $H(p_k,t_k)$, and $\mu_k$ for the pushforward under $F_k$ of the measure $f(y) \, dy|_\Omega$. The image of $\Omega$ under $F_k$ is bounded independently of $k$, and so the $\mu_k$ are all supported in some fixed compact set. Now to invoke part (c) of the theorem, we need only show $\mu_k \to \mu$ weakly. For this, consider a continuous bounded function $\psi(y)$, and observe that
\begin{align*}
\int_\Rn \psi \, d\mu_k
& = \int_\Rn \psi(F_k(y)) f(y) \, dy \\
& \to \int_\Rn \psi(F(y)) f(y) \, dy = \int_\Rn \psi \, d\mu 
\end{align*}
by using locally uniform convergence of $F_k$ to $F$, or else by dominated convergence.

\section{\bf Proof of \autoref{th:weinbergerallspace} --- center of mass for measures with unbounded support}  \label{sec:weinbergerproofallspace}

When the measure $\mu$ has unbounded support, the energy $E(x)$ used in proving \autoref{th:weinberger} could be  infinite for all $x$. Such unpleasantness will be avoided by subtracting $\Gamma(y)$ from the kernel and defining the \textbf{renormalized energy} 
\[
\mathcal{E}(x) = \int_\Rn \left( \Gamma(x+y)-\Gamma(y) \right) \, d\mu(y) , \qquad x \in \Rn .
\]
Formally, $\mathcal{E}(x)=E(x)-E(0)$, so that $\mathcal{E}$ may also be regarded as a ``relative energy''. 

We will show the renormalized energy is well defined, continuous, and differentiable with respect to $x$. The first step is to extend the renormalized kernel continuously to the sphere at infinity with respect to the $y$ variable.
\begin{lemma}[Extending the kernel to the sphere at infinity] \label{le:renormkerneleuclid}
Assume the limiting value $g(\infty) = \lim_{r \to \infty} g(r)$ exists and is finite. Suppose $x \to \widetilde{x} \in \Rn$ and $|y| \to \infty$. If $y/|y| \to \widetilde{y}$ for some unit vector $\widetilde{y}$, then
\begin{equation} \label{eq:kernelextend}
\Gamma(x+y)-\Gamma(y) \to g(\infty) \, \widetilde{x} \cdot \widetilde{y} .
\end{equation}
Hence the kernel
\[
K(x,r,\hat{y}) = 
\begin{cases}
\Gamma(x+r\hat{y})-\Gamma(r\hat{y}) , & x \in \Rn, \ \hat{y} \in S^{n-1}, \ r \in [0,\infty), \\
g(\infty) \, x \cdot \hat{y} , & x \in \Rn , \ \hat{y} \in S^{n-1}, \ r=\infty ,
\end{cases}
\]
is continuous and real valued on $\Rn \times [0,\infty] \times S^{n-1}$. In particular, $\Gamma(x+y)-\Gamma(y)$ is 
continuous and bounded whenever $x$ lies in a compact set and $y$ lies in $\Rn$. 
\end{lemma}
\begin{proof}
As $x \to \widetilde{x}, |y| \to \infty$ and $y/|y| \to \widetilde{y}$, one finds that
\[
|x+y| - |y| = \frac{|x+y|^2-|y|^2}{|x+y|+|y|} = \frac{|x|^2+2x \cdot y}{|x+y|+|y|} \to \widetilde{x} \cdot \widetilde{y} .
\]
Suppose to begin with $\widetilde{x} \cdot \widetilde{y} \neq 0$, so that (by the preceding formula) we may assume $|x+y| \neq |y|$ as we pass to the limit. By starting with the definition of $\Gamma$ and then multiplying and dividing by $|x+y| - |y|$ in order to get a mean value integral, we find
\begin{align*}
\Gamma(x+y) - \Gamma(y) 
& = \int_{|y|}^{|x+y|} g(s) \, ds \\
& = \big( |x+y| - |y| \big) \fint_{|y|}^{|x+y|} g(s) \, ds \\
& \to (\widetilde{x} \cdot \widetilde{y}) g(\infty) 
\end{align*}
since $|y| \to \infty$ and $|x+y| \to \infty$. If $\widetilde{x} \cdot \widetilde{y} = 0$ then the argument above continues to apply except for those values of $x,y$ such that $|x+y|=|y|$; but at those points we already have $\Gamma(x+y) - \Gamma(y) = 0 $, which is the desired limiting value. This completes the proof of the limit \eqref{eq:kernelextend}. 

The kernel $K$ is continuous with respect to all three variables when $r<\infty$. When $r=\infty$ the kernel is continuous with respect to $x$ and $\hat{y}$. Thus the only case remaining to check is when $x$ and $\hat{y}$ converge to points $\widetilde{x} \in \Rn$ and $\widetilde{y} \in S^{n-1}$ respectively and the finite value $r$ tends to infinity. Continuity in that case means that $K(x,r,\hat{y}) \to K(\widetilde{x},\infty,\widetilde{y}) = g(\infty) \widetilde{x} \cdot \widetilde{y}$, which is exactly the limit proved in \eqref{eq:kernelextend} with $y=r\hat{y}$. 

Finally, for each compact set $X \subset \Rn$ the kernel $K(x,r,\hat{y})$ is continuous on $X \times [0,\infty] \times S^{n-1}$, and so certainly the kernel is bounded there. 
\end{proof}
Boundedness of $\Gamma(x+y)-\Gamma(y)$ from \autoref{le:renormkerneleuclid} and the finiteness of $\mu$ together imply that the renormalized energy $\mathcal{E}(x)$ is well defined and finite-valued. Further, it depends continuously on $x \in \Rn$, by continuity of the kernel. Differentiation through the integral is justified similarly, since $g$ and thus $v=\nabla \Gamma$ are bounded, giving 
\begin{align*}
\nabla \mathcal{E}(x) 
& = \int_\Rn \nabla_{\! x} \left( \Gamma(x+y)-\Gamma(y) \right)  \, d\mu(y) \\
& = \int_\Rn v(x+y) \, d\mu(y) 
= V(x) .
\end{align*}
That is, critical points of the renormalized energy are zeros of $V$. To show the energy has a minimum point (hence a critical point), we will prove $\mathcal{E}(x) \to \infty$ as $|x| \to \infty$.

To do so, we develop two lower bounds on the renormalized kernel. The first estimate is the global worst-case.
\begin{lemma} \label{le:worsteuclid}
If $g$ is bounded for $r \geq 0$, then $\left| \Gamma(x+y)-\Gamma(y) \right| \leq (\sup |g|) |x|$ and hence
\[
\Gamma(x+y)-\Gamma(y) \geq - (\sup |g|) |x| , \qquad x ,y\in \Rn .
\]
\end{lemma}
\begin{proof}
\begin{align*}
\left| G(|x+y|) - G(|y|) \right| 
& = \left| \int_{|y|}^{|x+y|} g(s) \, ds \right| \\
& \leq (\sup |g|) \big| |x+y|-|y| \big| \leq (\sup |g|) |x| .
\end{align*}
\end{proof}
The second estimate, in the next lemma, provides a positive, uniform lower bound for $y$ in the ball $B(R)$ of radius $R$ centered at the origin. 
\begin{lemma} \label{le:bettereuclid}
Suppose the limiting value $g(\infty) = \lim_{r \to \infty} g(r)$ exists, and is positive and finite. If $R>0$ is fixed then  
for all sufficiently large $|x|$ we have 
\[
\Gamma(x+y)-\Gamma(y) > \frac{1}{2} g(\infty) |x| , \qquad y \in B(R) .
\]
\end{lemma}
\begin{proof}
Notice $g(r)$ is bounded, since it has a finite limit as $r \to \infty$. Because that limiting value is positive, a number $R^* > R$ exists such that $g(r) \geq (2/3)g(\infty)>0$ on $(R^*,\infty)$. Thus for $y \in B(R)$ and $|x| > R^*+R > 2R$, we have $|x+y|>|y|$ and so 
\begin{align*}
G(|x+y|) - G(|y|) 
& = \int_{|y|}^{|x+y|} g(s) \, ds \\
& \geq \int_{R^*}^{|x|-R} g(s) \, ds - (\sup |g|)(R^*-|y|) \\
& \geq \frac{2}{3} g(\infty) (|x|-R-R^*) - (\sup |g|) R^* \\
& > \frac{1}{2} g(\infty) |x|
\end{align*}
whenever $|x|$ is sufficiently large. 
\end{proof}
Now we can prove \autoref{th:weinbergerallspace}. 

\subsection*{Part (a) --- Existence.} Assume the limiting value $g(\infty) = \lim_{r \to \infty} g(r)$ exists, and is positive and finite, so that $g(\infty)>0$ and $g$ is bounded. Since $\mu\big(\Rn \setminus B(R)\big) \to 0$ and $\mu \big( B(R) \big) \to \mu(\Rn)>0$ as $R \to \infty$, we may fix $R$ large enough that 
\[
(\sup |g|) \mu\big(\Rn \setminus B(R)\big) < \frac{1}{4} g(\infty) \mu \big( B(R) \big) .
\]

As explained earlier in the section, the renormalized energy $\mathcal{E}(x)$ is finite-valued and differentiable, with gradient $\nabla \mathcal{E} = V$. To show $V$ vanishes somewhere, it is enough to prove $\mathcal{E}(x) \to \infty$ as $|x| \to \infty$, because then the energy has a minimum point. By decomposing $\Rn$ into the ball $B(R)$ and its complement, and estimating the renormalized kernel from below on those two sets by \autoref{le:bettereuclid} and \autoref{le:worsteuclid} respectively, we find for sufficiently large $|x|$ that 
\[
\mathcal{E}(x)
\geq \frac{1}{2} g(\infty) |x| \mu\big( B(R) \big) - (\sup |g|) |x|  \mu\big(\Rn \setminus B(R)\big) .
\]
Hence by choice of $R$ above, 
\[
\mathcal{E}(x)
\geq \frac{1}{4} g(\infty) \mu\big( B(R) \big) |x| ,
\]
which tends to infinity as $|x| \to \infty$. This completes the existence proof. 

\smallskip
\subsection*{Part (b) --- Uniqueness.} Assumptions (i) and (ii) each imply that the limiting value $g(\infty)$ is positive and finite, and so a vanishing point $x_c$ exists by part (a). 

The uniqueness of $x_c$ is proved by establishing strict convexity of the renormalized energy $\mathcal{E}(x)$, almost exactly as we did for the original energy $E(x)$ in the compactly supported case (\autoref{sec:weinbergerproof}). The only difference is that one must subtract $\Gamma(y)$ from the kernel before integrating to get the renormalized energy. This causes no difficulty for the proof, since the map $x \mapsto \Gamma(x+y)$ has exactly the same convexity properties as the renormalized map $x \mapsto \Gamma(x+y)-\Gamma(y)$. 

\smallskip
\subsection*{Part (c) --- Continuous dependence.} 
The hypotheses of parts (a) and (b) are satisfied for the measures $\mu$ and $\mu_k$, since part (c) assumes that either (i) holds or else (ii) holds for $\mu, \mu_1,\mu_2,\mu_3,\dots$. Let $x_c(\mu)$ be the unique minimum point of the renormalized energy $\mathcal{E}$ corresponding to the measure $\mu$, and $x_c(\mu_k)$ be the unique minimum point of the renormalized energy $\mathcal{E}_k$ corresponding to the measure $\mu_k$.  

The weak convergence $\mu_k \to \mu$ implies that $\mu_k(\Rn) \to \mu(\Rn)$, and so the measure $\mu_k(\Rn)$ is bounded independently of $k$. Hence the family $\{ \mathcal{E}_k \}$ is uniformly equicontinuous on $\Rn$, because boundedness of $v$ implies a bound on $\nabla \mathcal{E}_k=V_k$ that is independent of $k$. The weak convergence $\mu_k \to \mu$ and continuity and boundedness of $y \mapsto \Gamma(x+y)-\Gamma(y)$ (\autoref{le:renormkerneleuclid}) imply that $\mathcal{E}_k(x) \to \mathcal{E}(x)$ for each fixed $x \in \Rn$. A short argument using equicontinuity shows the convergence is uniform on each compact set $X \subset \Rn$. 

Let $\e>0$, and write $B$ for the open ball of radius $\e$ centered at $x_c(\mu)$. The strict minimizing property of $x_c(\mu)$ yields 
\[
\mathcal{E}(x_c(\mu)) < \min_{x \in \partial B} \mathcal{E}(x) ,
\]
and so (with $X=\partial B$) we get for all large $k$ that 
\[
\mathcal{E}_k(x_c(\mu)) < \min_{x \in \partial B} \mathcal{E}_k(x) .
\]
Hence $\mathcal{E}_k$ has a local minimum somewhere in the open ball $B$. This local minimum can occur only at the global minimum point $x_c(\mu_k)$, by strict convexity of the renormalized energy. Letting $\e \to 0$ now shows that $x_c(\mu_k) \to x_c(\mu)$ as $k \to \infty$, which is the desired continuous dependence.

\section{\bf Proof of \autoref{th:weinbergersigned} --- existence of center of mass for signed measures}  \label{sec:weinbergerproofsigned}

We need a two-sided version of the uniform bound in \autoref{le:bettereuclid}. 
\begin{lemma} \label{le:bettereuclidtwosided}
Suppose the limiting value $g(\infty) = \lim_{r \to \infty} g(r)$ exists, and is positive and finite. If $0<\e<1/2$ and $R>0$ are fixed and $|x|$ is sufficiently large, then   
\[
(1-2\e) g(\infty) |x| < \Gamma(x+y)-\Gamma(y) < (1+2\e)g(\infty) |x| , \qquad y \in B(R) .
\]
\end{lemma}
\begin{proof}
Notice $g$ is bounded. Since $g(r)$ converges to $g(\infty)$ as $r \to \infty$, we may choose $R^* > R$ such that 
\[
0 < (1-\e) g(\infty) \leq g(r) \leq (1+\e) g(\infty) < \infty , \qquad r \in (R^*,\infty) .
\] 
Then for $y \in B(R)$ and $|x| > R^*+R$, we have 
\begin{align*}
G(|x+y|) - G(|y|) 
& = \int_{|y|}^{|x+y|} g(s) \, ds \\
& \geq \int_{R^*}^{|x|-R} g(s) \, ds - (R^*-|y|) \sup |g| \\
& \geq (1-\e) g(\infty) (|x|-R-R^*) - R^* \sup |g| \\
& > (1-2\e) g(\infty) |x|
\end{align*}
whenever $|x|$ is sufficiently large. Similarly, one obtains an upper bound:
\begin{align*}
G(|x+y|) - G(|y|) 
& = \int_{|y|}^{|x+y|} g(s) \, ds \\
& \leq \int_{R^*}^{|x|+R} g(s) \, ds + (\sup |g|) R^* \\
& \leq (1+\e) g(\infty) (|x|+R-R^*) + (\sup |g|) R^* \\
& < (1+2\e) g(\infty) |x|
\end{align*}
whenever $|x|$ is sufficiently large. 
\end{proof}

To start proving \autoref{th:weinbergersigned}, notice the assumption $\mu(\Rn)>0$ implies $\mu^+(\Rn) > \mu^-(\Rn)$, and so we may choose $0<\e<1/2$ such that 
\[
(1-2\e) \mu^+(\Rn) > (1+2\e) \mu^-(\Rn) .
\]
Hence by fixing $R>0$ sufficiently large we can ensure that 
\begin{equation} \label{eq:Rchoice}
\begin{split}
& (1-2\e) g(\infty) \mu^+\big(B(R)\big) \\
& > (1+2\e) g(\infty) \mu^-\big(B(R)\big) + (\sup |g|) |\mu| \big(\Rn \setminus B(R)\big) + R^{-1} .
\end{split}
\end{equation}
As in the proof of \autoref{th:weinbergerallspace}, differentiating through the integral shows $\nabla \mathcal{E} = V$, and so to show the gradient $V$ vanishes somewhere, it is enough to prove $\mathcal{E}(x) \to \infty$ as $|x| \to \infty$. 

Decompose $\Rn$ into the ball $B(R)$ and its complement, and then estimate the renormalized energy integral from below as follows: for $y \in \Rn \setminus B(R)$ and $d\mu(y)$ use \autoref{le:worsteuclid}; for $y \in B(R)$ and $d\mu^+(y)$ use the lower bound from \autoref{le:bettereuclidtwosided}; and for $y \in B(R)$ and $-d\mu^-(y)$ use the upper bound from \autoref{le:bettereuclidtwosided}. The end result is that for sufficiently large $|x|$,  
\begin{align*}
\mathcal{E}(x)
& \geq \Big( (1-2\e) g(\infty) \mu^+\big( B(R) \big) - (1+2\e)g(\infty) \mu^-\big(B(R)\big) \\
& \qquad \qquad - (\sup |g|) |\mu|\big(\Rn \setminus B(R)\big) \Big) |x| \\
& \geq R^{-1} |x| \qquad \text{by choice of $R$ in \eqref{eq:Rchoice}} \\
& \to \infty
\end{align*}
as $|x| \to \infty$.

\section*{Acknowledgments}
This research was supported by a grant from the Simons Foundation (\#429422 to Richard Laugesen) and the University of Illinois Research Board (RB19045). I am grateful to Mark Ashbaugh and Jeffrey Langford for stimulating conversations and references about center of mass results. 

\bibliographystyle{plain}

\end{document}